\documentclass[11pt]{amsart}
\usepackage{amsfonts, amsmath, amssymb}
\usepackage{pb-diagram}
\usepackage{graphicx}
\usepackage[pdftex, pdfstartview=FitH]{hyperref}

\input xy
\xyoption{all}

\newcommand{\R}{{\mathbb R}}

\newcommand{\Graff}{{\rm Graff}}

\newcommand{\Aff}{{\rm Aff}}
\newcommand{\Ann}{{\rm Ann}}

\newcommand{\Length}{{\rm Length}}

\newtheorem{theorem}{Theorem}[section]
\newtheorem{corollary}[theorem]{Corollary}
\newtheorem{lemma}[theorem]{Lemma}
\newtheorem{proposition}[theorem]{Proposition}

\newtheorem{innertheorem}{Theorem}
\newenvironment{maintheorem}[1]
    {\renewcommand\theinnertheorem{#1}\innertheorem}
    {\endinnertheorem}

\theoremstyle{definition}
\newtheorem{definition}[theorem]{Definition}
\newtheorem*{definition*}{Definition}
\newtheorem{question}[theorem]{Question}

\title{Periodic Solutions of Hilbert's Fourth Problem}

\author{J.C. \'Alvarez Paiva}

\address{J.C. \'Alvarez Paiva, U.M.R. CNRS 8524
         U.F.R. de Math\'ematiques, 59655 Villeneuve d'Ascq C\'edex, France}

\email{alvarez@math.uni-lille1.fr}

\author{J. Barbosa Gomes}

\address{J. Barbosa Gomes, Departamento de Matem\'atica, Universidade Federal de Juiz de Fora, Juiz de Fora, MG, Brazil, 36036-900}

\email{barbosa.gomes@ufjf.edu.br}

\thanks{The first author was partially supported by the grant ANR12-BS01-0009}
\thanks{The second author was partially supported by CAPES Foundation - Brazil - BEX 002789/2015-09}



\begin{document}%

\begin{abstract} It is shown that a possibly irreversible $C^2$ Finsler metric on the torus, or on any other compact Euclidean space form, whose geodesics are straight lines is the sum of a flat metric and a closed $1$-form. This  is used to prove that if $(M,g)$ is a compact Riemannian symmetric space of rank greater than one and $F$ is a {\sl reversible} $C^2$ Finsler metric on $M$ whose unparametrized geodesics coincide with those of $g$, then $(M,F)$ is a Finsler symmetric space.
\end{abstract}
\maketitle
\begin{flushright}
    \begin{small}
        \parbox{4in} {{\sl The
theorem of the straight line as the shortest distance between
two points and the essentially equivalent theorem of Euclid
about the sides of a triangle, play an important part not
only in number theory but also in the theory of surfaces
and in the calculus of variations. For this reason, and
because I believe that the thorough investigation of the
conditions for the validity of this theorem will throw a new
light upon the idea of distance, as well as upon other
elementary ideas, {\it e.g.}, upon the idea of the plane, and
the possibility of its definition by means of the idea of
the straight line, the construction and systematic treatment of
the geometries here possible seem to me desirable. }} \\
        --- D. Hilbert
    \end{small}
\end{flushright}


\section{Introduction}

A natural amalgam of Hilbert's fourth and eighteenth problems is {\sl to study the continuous (possibly asymmetric) metrics on $\R^n$ for which straight lines are geodesics and which are invariant under the action of a Euclidean crystallographic group.} In other words, the problem asks to study the different possible {\sl metrizations} of compact Euclidean space forms: the metrics on these spaces for which the lines of the Euclidean structure are geodesics. 

Very roughly speaking, there is a natural identification between signed measures on the $n$-dimensional manifold of affine hyperplanes in $\mathbb{R}^n$ and symmetric {\sl projective metrics}, metrics for which straight lines are geodesics (see the next section for a short account of Busemann and Pogorelov's solution of Hilbert's fourth problem and~\cite{Alvarez:2003,Papadopoulos:2014} for a historical account). If to this we add the asymmetric projective metrics, of which many examples exist but whose general construction we still ignore, we see that the class of projective metrics is very large. Nevertheless, we have the following 
 
\begin{maintheorem}{I} \label{metric-version}
Let the distance function $d$ define a possibly asymmetric metric on $\R^n$ for which straight lines are geodesics. If the metric is invariant under a Euclidean crystallographic group and is derived from a $C^2$ Finsler metric, then there exists a possibly asymmetric norm $\| \cdot \|$ and a function $f$ on $\R^n$ such that 
$$
d(\mathbf{x},\mathbf{y}) = \| \mathbf{y - x} \| + f(\mathbf{y}) - f(\mathbf{x})
$$ 
for every pair of points $\mathbf{x}$, $\mathbf{y}$ in $\R^n$.
\end{maintheorem}

The differential-geometric formulation of this rigidity result is as follows:

\begin{maintheorem}{I${}^\prime$} \label{differential-geometric-version}
Let $(M,g)$ be a compact flat Riemannian manifold. If $F$ is a $C^2$ Finsler metric on $M$ whose unparametrized geodesics coincide with those of $g$, then it is the sum of a flat Finsler metric and a closed $1$-form. 
\end{maintheorem}

The regularity of the metric is crucial. There are many more examples of continuous projective metrics on tori. For instance, in~\cite[p. 221]{Busemann:1955} Busemann gives the following example of a continuous periodic metric on the plane for which straight lines are geodesics:
\begin{align*}
d((x_1,x_2),(y_1,y_2)) := & \sqrt{(x_1 - y_1)^2 + (x_2 - y_2)^2} \\
& + |7x_2 + \sin(2\pi x_2) - 7y_2 - \sin(2\pi y_2) | .
\end{align*}
Although we were not able to construct $C^1$ examples which are not of the form prescribed in
Theorems~\ref{metric-version} or~\ref{differential-geometric-version}, we have good reason to think that they exist.  

As we fleetingly mentioned earlier, the construction of all asymmetric projective metrics in $\mathbb{R}^n$ is still unknown. Nevertheless Theorem~\ref{metric-version} allows us to exhibit an interesting feature of asymmetric projective metrics. Namely, that if they originate from a $C^2$ Finsler metric, they are determined, modulo a differential, by what happens ``at infinity". 

\begin{maintheorem}{II} \label{determination-at-infinity}
If two $C^2$ projective Finsler metrics defined on an open convex subset $\mathcal{O} \subset \mathbb{R}^n$ agree outside of some compact set, then their difference is the differential of a $C^3$ function on $\mathcal{O}$. 
\end{maintheorem}

Another important consequence of Theorem~\ref{differential-geometric-version} is the following 

\begin{maintheorem}{III} \label{symmetric-spaces}
Let $(M,g)$ be a compact Riemannian symmetric space of rank greater than one. If $F$ is a reversible $C^2$ Finsler metric on $M$ whose unparametrized geodesics coincide with those of $g$, then $F$ is invariant under the geodesic symmetries of $(M,g)$. In other words, $(M,F)$ is a symmetric Finsler space. 
\end{maintheorem}

Note that for reversible $C^2$ Finsler metrics this is the solution of the extension of Hilbert's fourth problem to compact symmetric spaces of rank greater than one. For rank one spaces this problem is completely open except in the case of spheres and their quotients (see the discussion around Problem~11 
in~\cite[p.22]{Alvarez:survey}). 

A notable corollary of Theorem~\ref{symmetric-spaces} is that {\it if $F$ is a $C^2$ reversible Finsler metric on compact, connected Lie group $G$ which is neither $SU(2)$ nor $SO(3)$ and $F$ has the same unparametrized geodesics as a bi-invariant Riemannian metric, then it is a bi-invariant metric.} 

All of the theorems stated above generalize from Finsler metrics to $1$-densities, also known as parametric integrands of degree one. Like Finsler metrics, these are Lagrangians that are homogeneous of degree one in the velocities, but no positivity nor convexity requirement is made on the integrand. We have cast the statements of the main results in the Finsler setting not only because Hilbert's fourth problem was the main motivation in these researchers, but because our proof of Theorems~\ref{metric-version} and~\ref{differential-geometric-version} uses the convexity of the Finsler metrics, and the Finsler case leads directly to the more general result by a simple trick. On the other hand, the extension to $1$-densities allows us to cast Theorem~\ref{metric-version} as a characterization of those solutions to Hamel's system of differential equations that satisfy certain smoothness and symmetry conditions:

\begin{maintheorem}{I${}^{\prime \prime}$}\label{result-for-densities}
Let $L(\mathbf{x}, \mathbf{v})$ be a continuous, real-valued function on $\mathbb{R}^n \times  \mathbb{R}^n $ that is $C^2$ on $\mathbb{R}^n \times  \mathbb{R}^n \! \setminus \! \{\mathbf{0}\}$ and which, on this open set, satisfies the system of differential equations
$$
\frac{\partial^2 L}{\partial x_i \partial v_j} = \frac{\partial^2 L}{\partial x_j \partial v_i} \ \ 
(1 \leq i, j \leq n) .
$$
If $L(\mathbf{x} + \mathbf{m}, \lambda \mathbf{v}) = \lambda L(\mathbf{x}, \mathbf{v})$ for all  $\lambda > 0$ and all integer vectors $\mathbf{m} \in \mathbb{Z}^n$, then there exists a $C^3$ function
$f : \mathbb{R}^n \rightarrow \mathbb{R}$ such that $L(\mathbf{x}, \mathbf{v}) - df(\mathbf{x})\cdot \mathbf{v}$
is independent of the $\mathbf{x}$ variable. 
\end{maintheorem}

The plan of the paper is as follows. In Section~\ref{History-section} we recall previous work on Hilbert's fourth problem and, in passing, prove a weaker version of Theorem~\ref{metric-version} and~\ref{differential-geometric-version}, where we assume the Finsler metric is smooth and reversible. 

The two-dimensional case of Theorems~\ref{metric-version} and~\ref{differential-geometric-version} is proved in Section~\ref{2D-case-section}. In Section~\ref{Minowski-Randers-section} we consider {\sl Minkowski-Randers} metrics. These are Finsler metrics on $\mathbb{R}^n$ which are the sum of a 
translation-invariant metric and a $1$-form. The main result here is that a Finsler metric is
a Minkowski-Randers metric if and and only if it induces Minkowski-Randers metrics on every two-dimensional affine subspace. This result is used in Section~\ref{Applications-section} to prove Theorems~\ref{metric-version} and~\ref{differential-geometric-version}, which are later applied to prove Theorems~\ref{determination-at-infinity} and~\ref{symmetric-spaces}. 

\medskip
\noindent {\bf Acknowledgments.} The authors thank K.~Tzanev and V.~Matveev for helpful 
discussions. Most of the work in this paper was done while the second author enjoyed a postdoctoral stay
at the Universit\'e de Lille~1 and he is thankful for its hospitality during this time.

\section{Background on Hilbert's fourth problem} \label{History-section}

As the fourth of his famous list of problems delivered at the International Congress of Mathematicians held in Paris in the year 1900, Hilbert asked for the construction and systematic treatment of all geometries for which the straight line is the shortest distance between two points. The original wording is somewhat vague, but the modern interpretation, due to Busemann~\cite{Busemann:Hilbert}, is a clear and enticing problem: {\sl to construct and study all continuous, possibly asymmetric metrics defined on open, convex subsets of real projective spaces for which geodesics lie on projective lines.}  If we leave the real projectives spaces aside and consider an arbitrary open, convex subset $\mathcal{O} \subset \mathbb{R}^n$, the problem translates into constructing and studying all continuous functions 
$$
d :  \mathcal{O} \times  \mathcal{O} \longrightarrow  \mathbb{R}
$$
satisfying the properties
\begin{itemize}
\item $d(\mathbf{x}, \mathbf{y}) \geq 0$.
\item $d(\mathbf{x}, \mathbf{y}) = 0$ if and only if $\mathbf{x} = \mathbf{y}$.
\item $d(\mathbf{x}, \mathbf{y}) + d(\mathbf{y}, \mathbf{z}) \geq d(\mathbf{x}, \mathbf{z})$.
\item $d(\mathbf{x}, \mathbf{y}) + d(\mathbf{y}, \mathbf{z}) = d(\mathbf{x}, \mathbf{z})$ whenever
$\mathbf{y}$ belongs to the line segment $\mathbf{xz}$.
\end{itemize}

In this generality the problem remains unsolved, but if we require that the metric be symmetric (i.e., $d(\mathbf{x}, \mathbf{y}) = d(\mathbf{y}, \mathbf{x})$), then Busemann and Pogorelov have provided a delightful solution (see~\cite{Busemann:Hilbert} and \cite{Pogorelov:1978}), which we proceed to describe.

Let $\Graff_{n-1}(\mathcal{O})$ be the manifold of all affine hyperplanes intersecting the open, convex set $\mathcal{O} \subset \mathbb{R}^n$, and let $\mu$ be a signed Borel measure on $\Graff_{n-1}(\mathcal{O})$ satisfying the following two conditions: 
\begin{enumerate}
\item The measure of the set of hyperplanes passing through any given point in $\mathcal{O}$ is zero.
\item If $\mathbf{x}$, $\mathbf{y}$, and $\mathbf{z}$ $(\mathbf{x} \neq \mathbf{y, z})$ are any three points in $\mathcal{O}$, the set of hyperplanes separating the point $\mathbf{x}$ from the segment  $\mathbf{yz}$ has strictly positive measure.
\end{enumerate}
We follow~\cite{Alvarez-Gelfand-Smirnov} in calling such measures {\sl quasipositive}, and remark in passing that it is not clear how to construct them in general---although in two-dimensions they are just the positive Borel measures satisfying~(1). 

Busemann's idea is to use a reverse form of Crofton's formula in integral geometry: given a quasipositive measure $\mu$ on  $\Graff_{n-1}(\mathcal{O})$, define the distance function $d_\mu(\mathbf{x}, \mathbf{y})$ as the measure of the set of all hyperplanes intersecting the segment $\mathbf{xy}$. It is easily verified
that the resulting distance is continuous, symmetric, and projective. Pogorelov's contribution was to show that {\sl in two dimensions every continuous, symmetric, projective metric is obtained from Busemann's construction. In higher dimensions it can be obtained as a limit of the metrics constructed by Busemann.} 
If we want to avoid this passage to the limit, Szab\'o in~\cite{Szabo:1986} showed that one can, and must, extend Busemann's construction from measures to certain classes of distributions. 

The central concept in Pogorelov's solution and in most of this paper is that of Finsler space.

\begin{definition}
A function $\| \cdot \| : V \rightarrow [0, \infty)$ on a finite-dimensional real vector space $V$ is a {\sl Minkowski norm} if it is $C^2$ outside the origin and
\begin{itemize}
\item $\| \mathbf{x} \| = 0$ if and only if $\mathbf{x} = \mathbf{0}$.
\item $\| \lambda \mathbf{x} \| = \lambda \| \mathbf{x} \| $ for any real positive number $\lambda $.
\item $\| \mathbf{x+y}\| \leq \|\mathbf{x}\| +  \|\mathbf{y}\|$
\item The Hessian, computed in any affine system of coordinates, of the function $\mathbf{x} \mapsto  \|\mathbf{x}\|^2$ is positive-definite outside the origin. {\sl In particular, the unit sphere in $V$ and that of its dual normed space $V^*$ are strictly convex.}
\end{itemize}
\end{definition}

\begin{definition}
A  $C^k$ $(k \geq 2)$ {\sl Finsler metric} on a smooth manifold $M$ is a continuous function $F : TM \longrightarrow [0, \infty)$ that is $C^k$ outside the zero section and for which its restriction to any tangent space is a Minkowski norm.  
\end{definition}

Given a Finsler metric on manifold $M$ we define the length of a smooth parametrized curve 
$\gamma:[a,b] \longrightarrow M$  as
$$
\Length(\gamma) := \int_a^b F(\dot{\gamma}(t)) \, dt \ ,
$$
and the distance between two points $\mathbf{x}$ and $\mathbf{y}$ in the same path connected component of $M$
as the infimum of the lengths of all smooth curves starting at  $\mathbf{x}$ and ending at $\mathbf{y}$. The metric will be symmetric (i.e., $d(\mathbf{x}, \mathbf{y}) =  d(\mathbf{y}, \mathbf{x})$) if only if the Finsler metric is {\sl reversible:} $F(\mathbf{v_x}) = F(\mathbf{-v_x})$. 

Projective Finsler metrics play an important role in Hilbert's fourth problem because every continuous, symmetric, projective metric is a limit of projective reversible Finsler metrics (see \cite[pp.~31--45]{Pogorelov:1978}). Pogorelov's key theorem (as presented, in 
\cite{Szabo:1986, Alvarez-Gelfand-Smirnov, Alvarez-Fernandes:2007}) is the following

\begin{theorem}[Pogorelov~\cite{Pogorelov:1978}]\label{Pogorelov-theorem}
If $\mu$ is a smooth quasipositive measure on the manifold $\Graff_{n-1} (\mathbb{R}^n)$ of all affine hyperplanes, the metric $d_\mu$ obtained through Busemann's construction is the distance function of a smooth, reversible projective Finsler metric on $\mathbb{R}^n$. Conversely, if $F$ is a smooth, reversible projective Finsler metric on $\mathbb{R}^n$, then there exists a unique smooth quasipositive measure $\mu$ on 
$\Graff_{n-1} (\mathbb{R}^n)$ such that $d_\mu$ is the distance associated to $F$.
\end{theorem}

Using Pogorelov's theorem it is relatively easy to prove the following weak version of Theorem~\ref{metric-version}:

\begin{theorem}\label{weak-metric-theorem}
Let $d : \mathbb{R}^n \times \mathbb{R}^n \rightarrow [0, \infty)$ be a symmetric, projective 
metric derived from a smooth Finsler metric on $\mathbb{R}^n$. If $d$ is invariant under the action of some 
Euclidean crystallographic group, then there exists a (symmetric) norm $\| \cdot \|$ on $\mathbb{R}^n$ such that 
$$
d(\mathbf{x}, \mathbf{y}) = \| \mathbf{y - x} \|
$$
for every pair of points $\mathbf{x}$, $\mathbf{y}$ in $\mathbb{R}^n$.
\end{theorem}

Theorem~\ref{Pogorelov-theorem} establishes a bijection between a class of Finsler spaces and a class of measure spaces. A simple, albeit important remark is that this bijection is equivariant with respect to the natural actions of the group of invertible affine transformations of $\mathbb{R}^n$, $\Aff(\mathbb{R}^n)$.

The left action $(T,\mathbf{x}) \mapsto T(\mathbf{x})$ of $\Aff(\mathbb{R}^n)$ on $\mathbb{R}^n$ induces the right action $(T,F) \mapsto T^*F$ on the space of reversible projective metrics. Likewise, the  left action $(T,\zeta) \mapsto T(\zeta)$ of the affine group on the affine Grassmannian of hyperplanes in
$\mathbb{R}^n$ induces the right action $(T,\mu) \mapsto T^{-1}_\# \! \mu$ on the space of smooth, positive measures on $\Graff_{n-1}(\mathbb{R}^n)$. Here $ T^{-1}_\# \! \mu$ is the push-forward of the measure $\mu$ by the transformation 
$$
T^{-1} : \Graff_{n-1}(\mathbb{R}^n) \longrightarrow \Graff_{n-1}(\mathbb{R}^n). 
$$

\begin{lemma}\label{equivariance}
In the bijection established by Theorem~\ref{Pogorelov-theorem}, if the Finsler metric $F$ on $\mathbb{R}^n$ corresponds to the quasipositive measure $\mu$ on $\Graff_{n-1}(\mathbb{R}^n)$, then the metric $T^*F$ corresponds to the measure $T^{-1}_\# \! \mu$ for any given invertible affine transformation
$T \in \Aff(\mathbb{R}^n)$, and vice-versa.
\end{lemma}

\begin{proof}
It suffices to note that the $\mu$-measure of all hyperplanes intersecting the segment $T(\mathbf{xy})$
is the same as the $T^{-1}_\# \! \mu$-measure of all hyperplanes intersecting the segment $\mathbf{xy}$. In other words, $T^* d(\mathbf{x}, \mathbf{y}) = d_{T^{-1}_\# \! \mu} (\mathbf{x}, \mathbf{y})$. The lemma follows at once from the relation between the distance and the Finsler metric. 
\end{proof}

Because of this result, if a reversible, projective Finsler metric is invariant under the action of a Euclidean crystallographic group $\Gamma$, then so is the corresponding measure. However, while the action of $\Gamma$ on $\mathbb{R}^n$ is properly discontinuous, its induced action on the affine Grassmannian $\Graff_{n-1}(\mathbb{R}^n)$ is a very different kind of animal. 

\begin{lemma}\label{action on Graff}
If a continuous function $f : \Graff_{n-1}(\mathbb{R}^n) \rightarrow \R$ is invariant under the translation of hyperplanes by integer vectors, then it is invariant under all translations. As a consequence, any smooth measure on $ \Graff_{n-1}(\mathbb{R}^n)$ that is invariant under translation by integer vectors is also invariant under all translations. 
\end{lemma}

\begin{proof}
The result will follow immediately once we establish that for a dense subset of $\Graff_{n-1}(\mathbb{R}^n)$ the $\mathbb{Z}^n$-orbit of any hyperplane $Y$ in the subset is dense in the $\mathbb{R}^n$-orbit of $Y$.

In order to best express the action of $\mathbb{R}^n$ (and $\mathbb{Z}^n$) on the Grassmannian of affine hyperplanes, we pass to its double cover, the Grassmannian of cooriented affine hyperplanes, which we identify with the product 
$S^{n-1} \times \mathbb{R}$ by associating the hyperplane with equation 
$\mathbf{u \cdot x} = p$ to the pair $(\mathbf{u}, p) \in S^{n-1} \times \mathbb{R}$.

The action of $\mathbb{R}^n$ on the Grassmannian of cooriented affine hyperplanes is
given by 
$$
\left( \mathbf{x}, (\mathbf{u}, p) \right) \longmapsto 
(\mathbf{u}, p + \mathbf{u \cdot x}) .
$$

{\sl The proof of the lemma is thus reduced to showing that for a dense set of points in the 
unit sphere $S^{n-1}$ the set $\{\mathbf{u \cdot m} : \mathbf{m} \in \mathbb{Z}^n \}$ is dense in $\mathbb{R}$.} In fact, this last set is dense when $\mathbf{u}$ is a unit vector such that $\mathbf{u \cdot m} \neq 0$ for all nonzero integer vectors.

Indeed, given a vector  $\mathbf{x}$ orthogonal to $\mathbf{u}$ and a positive number $\epsilon$, either  Dirichlet's approximation theorem or Minkowski's first theorem in the geometry of numbers implies that some
multiple of $\mathbf{x}$, let's say $\lambda \mathbf{x}$, is $\epsilon$-close to a nonzero integer vector
$\mathbf{n}$. It follows that
$$
0 < |\mathbf{u \cdot n}| = |\mathbf{u \cdot (n - \lambda x + \lambda x)}| = 
|\mathbf{u \cdot (n - \lambda x)}| \leq \epsilon . 
$$
This implies that the set $$
\{\mathbf{u} \cdot k\mathbf{n} : k \in \mathbb{Z} \} \subset \{\mathbf{u \cdot m} : \mathbf{m} \in \mathbb{Z}^n \}
$$
is an $\epsilon$-net of $\mathbb{R}$ and, since $\epsilon$ was arbitrary, that 
$\{\mathbf{u \cdot m} : \mathbf{m} \in \mathbb{Z}^n \}$ is dense.

Now, all we have to show is that the set of unit vectors $\mathbf{u}$ such that 
$\mathbf{u \cdot m} \neq 0$ for all nonzero integer vectors is dense on the sphere. In fact,
it is a residual set: its complement is the (countable) union of the closed, nowhere dense sets
$$
\{\mathbf{v} \in S^{n-1} : \mathbf{v \cdot m} = 0 \}
$$
as $\mathbf{m}$ ranges over all nonzero integer vectors. 
\end{proof}

\begin{proof}[Proof of Theorem~\ref{weak-metric-theorem}]
By Biebarbach's first theorem, any $n$-dimensional crystallographic group $\Gamma$ contains $n$ linearly independent translations. Up to conjugation by an affine transformation, we may assume that $\Gamma$ contains $\mathbb{Z}^n$.

Pogorelov's Theorem and Lemma~\ref{equivariance} tell us that if $F$ is a smooth, reversible projective Finsler metric on $\mathbb{R}^n$ that is invariant under $\mathbb{Z}^n$, then its associated smooth quasipositive measure $\mu$ on $\Graff_{n-1}(\mathbb{R}^n)$ is also invariant under $\mathbb{Z}^n$. Lemma~\ref{action on Graff} implies that $\mu$ is invariant under all translations and, using Lemma~\ref{equivariance} again, we conclude that so is $F$. This means that
$F(\mathbf{x}, \mathbf{v}) = \| \mathbf{v} \|$ for some Minkowski norm, and that the induced distance between any two points $\mathbf{x}$ and $\mathbf{y}$ equals $\| \mathbf{y - x} \|$.
\end{proof}

\section{Two-dimensional case} \label{2D-case-section}

Unfortunately, the simple proof of Theorem~\ref{weak-metric-theorem} cannot be extended to a proof of Theorem~\ref{metric-version}. Not only does the integral-geometric approach fail to yield non-reversible projective Finsler metrics, but for $n > 2$ continuous quasipositive measures on the affine Grassmannian $\Graff_{n-1}(\mathbb{R}^n)$ are not enough to yield all $C^2$ reversible projective Finsler metrics on $\mathbb{R}^n$. This last point is explained in great detail by Szab\'o in~\cite{Szabo:1986}. 

In this section we switch to an approach that relies on very simple ideas in variational calculus and symplectic geometry, and which applies equally well to irreversible projective Finsler metrics. In fact, we only need three basic facts: (1) a projective Finsler metric satisfies a specific system of linear partial differential equations; (2) a projective Finsler metric induces a symplectic form on the space of oriented lines which characterizes the metric up to the addition of a differential, and (3) most aspects of Hilbert's fourth problem are two-dimensional in nature because a projective Finsler metric induces projective Finsler metrics on all its $2$-dimensional affine subspaces. 

\subsection{Hamel's equations and the Hilbert forms}
One of the key elements in the proof of Pogorelov's Theorem (Theorem~\ref{Pogorelov-theorem}) is Hamel's 
result~(\cite{Hamel:1903}) that projective Finsler metrics satisfy a system of partial differential equations. Since, by linearity, the equations hold for differences of projective Finsler metrics---which may no longer be Finsler metrics---it is useful to consider a more general setting. 

\begin{definition}
A differential $1$-density (or a $1$-density for short) of class $C^k$ $(k \geq 0)$ on a manifold $M$ is
a continuous function $L : TM \rightarrow \mathbb{R}$ that is homogeneous of degree one (i.e., 
$L(\lambda \mathbf{v}_\mathbf{x}) = \lambda L(\mathbf{v}_\mathbf{x})$ for all $\lambda \geq 0$) and $C^k$ outside the zero section. 
\end{definition}

Like Finsler metrics, $1$-densities can be integrated along oriented piecewise $C^1$ curves independently of their  parametrization, and if a $1$-density $L$ is $C^2$ it makes sense to say it is {\sl projective} when all parametrized lines $t \mapsto \mathbf{x} + t\mathbf{v}$ solve the Euler-Lagrange equations 
$$
\frac{d}{dt} \frac{\partial L}{\partial \mathbf{v}} - \frac{\partial L}{\partial \mathbf{x}} = 0 .
$$
This condition leads directly to the following

\begin{theorem}[Hamel~\cite{Hamel:1903}]
A $C^2$ $1$-density $L : T \mathbb{R}^n \longrightarrow [0,\infty)$ is projective if and only if outside the zero section it satisfies the system of partial differential equations
$$
\frac{\partial^2 L}{\partial x_i \partial v_j} = \frac{\partial^2 L}{\partial x_j \partial v_i}
$$
for all $1 \leq i, j, \leq n$. 
\end{theorem}

Given a $C^2$ $1$-density $L : T\mathbb{R}^n \longrightarrow \mathbb{R}$, its {\sl Hilbert $1$-form} is 
defined as
$$
\alpha_{{}_L} := \sum_{i=1}^n \frac{\partial L}{\partial v_i} \, dx_i .
$$
This form is homogeneous of degree zero and hence it lives naturally on the bundle of tangent rays,
$\mathcal{S} T\mathbb{R}^n$, defined as the set of equivalence classes of nonzero tangent vectors up
to positive multiples. If $L$ is a Finsler metric, we can identify the bundle of tangent rays with 
the unit sphere bundle
$$
\mathcal{S}_{{}_L}  T\mathbb{R}^n = 
\{(\mathbf{x}, \mathbf{v}) \in T\mathbb{R}^n : L(\mathbf{x}, \mathbf{v}) = 1 \} .
$$
From the viewpoint of the Hilbert $1$-form, the basic difference between Finsler metrics and general $1$-densities is that in the former case $\alpha_{{}_L}$ is a contact $1$-form on the bundle of tangent rays. This is the same as saying that {\sl the Hilbert $2$-form} $\omega_{{}_L} := d\alpha_{{}_L}$ is a form of maximal rank in $\mathcal{S}T\mathbb{R}^n$. When this is so, the kernel of $\omega_{{}_L}$ is a line field and the geodesic 
vector field is the Reeb vector field $\mathcal{R}_{\alpha_{{}_L}}$ defined by the equations
$$
\alpha_{{}_L}(\mathcal{R}_{\alpha_{{}_L}}) = 1 \textrm{ and } 
\omega_{{}_L}(\mathcal{R}_{\alpha_{{}_L}}, \cdot) = 0 .
$$

Notice that in coordinates we have that 
$$
\omega_{{}_L} = \sum_{i,j=1}^n \frac{\partial^2 L}{\partial x_i \partial v_j} \, dx_i \wedge dx_j +
 \sum_{i,j=1}^n \frac{\partial^2 L}{\partial v_i \partial v_j} \, dv_i \wedge dx_j .
$$
Therefore, it follows from Hamel's equations that {\sl if the $1$-density $L$ is projective, then } 
$$
\omega_{{}_L} =  \sum_{i,j=1}^n \frac{\partial^2 L}{\partial v_i \partial v_j} \, dv_i \wedge dx_j .
$$
We wish to prove---and this is a re-statement of Theorem~\ref{result-for-densities}---that if $L$ is also periodic (i.e., invariant under translations by integer vectors), then the partial derivatives 
$$
\frac{\partial^2 L}{\partial v_i \partial v_j}(\mathbf{x},\mathbf{v}), \ \ 1 \leq i,j \leq n ,
$$
are independent of the $\mathbf{x}$ variable. We shall proceeed to do so in the case $n = 2$ and assuming that $L$ is a Finsler metric.

\subsection{Periodic projective Finsler metrics on the plane}

\begin{theorem}\label{2D-case}
A $C^2$ Finsler metric on the plane that is both projective and periodic is the sum of a Minkowski norm and the  differential of a $C^3$ function on the plane. 
\end{theorem}

\begin{proof}
Let us consider the correspondence
$$
\xymatrix{ 
      & \mathcal{S} T \mathbb{R}^2  \ar[dl]_{\pi_1}\ar[dr]^{\pi_2} &   \\
      \mathbb{R}^2  &     &  \Graff_1^+(\mathbb{R}^2) ,}
$$
where $\pi_1(\mathbf{x},\mathbf{v}) = \mathbf{x}$ and $\pi_2(\mathbf{x},\mathbf{v})$ is the oriented line passing
through $\mathbf{x}$ with direction vector $\mathbf{v}$. Note that the correspondence is equivariant with respect to the natural actions of the affine group  $\Aff(\mathbb{R}^2)$ on the three manifolds. 

If $F$ is a $C^2$ projective Finsler metric on the plane the tangent spaces to the fibers of the fibration 
$$
\pi_2 : \mathcal{S} T \mathbb{R}^2 \longrightarrow \Graff_1^+(\mathbb{R}^2)
$$ 
form the line field $\ker \omega_{{}_F}$. It follows that there exists a unique nowhere-zero $2$-form
${\widehat{\omega}}_{{}_F}$  on $\Graff_1^+(\mathbb{R}^2)$ so that 
$\pi_2^* {\widehat{\omega}}_{{}_F} = \omega_{{}_F}$. Morever, because of the definition of ${\widehat{\omega}}_{{}_F}$ and the equivariance of $\pi_2$, the form $ \omega_{{}_F}$ is invariant under  an affine transformation if and only if ${\widehat{\omega}}_{{}_F}$ is invariant. 

By hypothesis, the metric $F$ is invariant under translations by integer vectors, which immediately implies
that the continuous $2$-forms $\omega_{{}_F}$ and ${\widehat{\omega}}_{{}_F}$ are also $\mathbb{Z}^2$-invariant. By Lemma~\ref{action on Graff}, the area form ${\widehat{\omega}}_{{}_F}$ must be invariant under all translations, and therefore the same is true for $\omega_{{}_F}$.  In other words, the partial derivatives appearing in the expression
$$
\omega_{{}_F} =  \sum_{i,j=1}^2 \frac{\partial^2 F}{\partial v_i \partial v_j} \, dv_i \wedge dx_j 
$$
are independent of the $\mathbf{x}$ variable. 

We will conclude the proof by showing that this implies that $F$ is the sum of a translation-invariant Finsler metric and a closed $1$-form. In order to do this, define the translation-invariant Finsler metric 
$F_0(\mathbf{x},\mathbf{v}) := F(\mathbf{0}, \mathbf{v})$ and consider the $1$-density $L := F - F_0$. Because of
the translation-invariance of $\omega_{{}_F}$, we have that
$$
 \sum_{i,j=1}^2 \frac{\partial^2 L}{\partial v_i \partial v_j} \, dv_i \wedge dx_j  = 0 ,
$$
which tells us that $L$ is linear in the velocities. In other words, $L$ is a differential $1$-form of class
$C^2$. On the other hand, Hamel's equations,
$$
\sum_{i,j=1}^n \frac{\partial^2 L}{\partial x_i \partial v_j} \, dx_i \wedge dx_j = 0 ,
$$
tell us that it is a closed $1$-form. Thus $F$ is the sum of the Minkowski norm $F_0$ and the differential of a $C^3$ function on the plane.
\end{proof}

\section{Minkowski-Randers spaces} \label{Minowski-Randers-section}

In the previous section we have shown that if a $C^2$ projective Finsler metric on the plane is invariant under translations by a lattice, then it is the sum of a translation-invariant metric and a closed $1$-form. The next step towards the proof of Theorems~\ref{metric-version} and~\ref{differential-geometric-version} is to show that if a Finsler metric on $\mathbb{R}^n$ induces this kind of metric on a dense 
 subset of 2-flats, then the metric on $\mathbb{R}^n$ is itself the sum of a translation-invariant metric and a closed $1$-form. In this section, we are able to do this thanks to an interesting result by Groemer (\cite{Groemer:1997}) in convex geometry.

\begin{theorem}[Groemer]\label{Groemer}
Let $(X, \langle \cdot , \cdot \rangle)$ be finite-dimensional Euclidean space of dimension greater than two, let $\mathcal{P}$ and $\mathcal{Q}$ be two convex bodies in $X$ one of which is strictly convex, and let $W \subset X$ be a $1$-dimensional subspace. If every orthogonal projection of  $\mathcal{P}$ onto a hyperplane containing $W$ is a translate of the corresponding orthogonal projection of $\mathcal{Q}$, then  $\mathcal{P}$ is a translate of $\mathcal{Q}$. 
\end{theorem}

In fact, we shall derive from Groemer's theorem a result in Finsler geometry that is of independent interest. 

\begin{definition}
A {\sl Minkowski-Randers space} is an affine space provided with a $C^2$ Finsler metric that is the sum of a translation-invariant metric and a $1$-form. 
\end{definition}

\begin{theorem}\label{Minkowski-Randers-characterization}
An $n$-dimensional affine space $V$ provided with a Finsler metric is a Minkowski-Randers space if and only if every affine hyperplane in $V$ provided with the induced metric is itself a Minkowski-Randers space.
\end{theorem}

A key point in the proof is the geometric characterization of Minkowski-Randers metrics as those Finsler metrics on an affine space $V$ for which the unit co-discs are related by translations. 

\begin{proposition}\label{geometric-Minkowski-Randers}
Given a Minkowski-Randers space $(V,F)$ and any two points $\mathbf{x}$ and $\mathbf{y}$ in $V$, the unit co-disc $B_\mathbf{x}^* \subset T_\mathbf{x}^* V$ and the pullback of the unit co-disc $B_\mathbf{y}^* \subset T_\mathbf{y}^* V$ by the translation $\tau_\mathbf{y-x}$ are translates in $T_\mathbf{x}^*V$. Conversely, any $C^2$ Finsler metric on an affine space $V$ satisfying this property is Minkowski-Randers metric. 
\end{proposition}

\begin{proof}
It is sufficient to remark that adding a $1$-form $\beta$ to an arbitrary Finsler metric on a manifold $M$ is tantamount to displacing the co-disc in each cotangent space 
$T_\mathbf{x}^* M$ by the covector $\beta(\mathbf{x})$.  
\end{proof}

Proposition~\ref{geometric-Minkowski-Randers} makes it even clearer that Minkowski-Randers metrics are preserved under invertible affine transformations and that if a sequence of  Minkowski-Randers metrics converge uniformly on compact subsets to a $C^2$ Finsler metric, the limit is also a Minkowski-Randers metric. 

In order to apply Groemer's result to prove Theorem~\ref{Minkowski-Randers-characterization}
we first transform it from its apparent Euclidean form into another that is manifestly equivariant under the action of the linear group. 

\begin{theorem}\label{invariant-Groemer}
Let $X$ be a finite-dimensional vector space of dimension greater than two, let $\mathcal{P}$ and $\mathcal{Q}$ be two convex bodies in $X$ which contain the origin as an interior point, and let $W \subset X$ be a $1$-dimensional subspace. If the dual of $\mathcal{P}$, $\mathcal{P}^*$, is strictly convex and for every hyperplane $Y$ containing $W$ the convex bodies $(\mathcal{P} \cap Y)^*, (\mathcal{Q} \cap Y)^* \subset Y^*$ are translates, then $\mathcal{P}^*$ and $\mathcal{Q}^*$ are also translates. 
\end{theorem}

Recall that the {\sl dual of a convex body} $\mathcal{P} \subset X$ which contains the origin as an interior point is the convex body in the dual space $X^*$ defined as
$$
\mathcal{P}^* := \{ {\boldsymbol\xi} \in X^* : {\boldsymbol\xi}(\mathbf{v}) \leq 1 \text{ for all } \mathbf{v} \in P \} .
$$

The equivalence between Theorems~\ref{Groemer} and~\ref{invariant-Groemer} is essentially a tautology although unravelling it takes some ink.  We start with two basic remarks:

\begin{lemma}\label{remark_1}
In the diagram
 $$
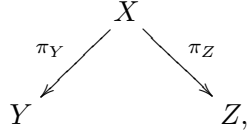
\xymatrix{
      &  X \ar[dl]_{\pi_Y}\ar[dr]^{\pi_Z} &   \\
      Y  &     &  Z,}
$$
let $X$,$Y$, and $Z$ be finite-dimensional vector spaces, and let $\pi_Y$ and $\pi_Z$ be surjective linear maps.
If the kernels of these two maps coincide, then the images of any two subsets of $X$ under $\pi_Y$ are translates if and only if their images under $\pi_Z$ are translates. 
\end{lemma}

\begin{proof}
Simply notice that the map that takes a point $\mathbf{y} \in Y$ to the unique element in
the singleton $\pi_Z(\pi_Y^{-1}\{ \mathbf{y} \})$ defines a linear isomorphism between $Y$ and $Z$.
\end{proof}

\begin{lemma}\label{remark_2}
Let $X$ be a finite-dimensional vector space and let $\mathcal{P} \subset X$ be a convex body which 
contains the origin as an interior point. If $Y \subset X$ is a linear subspace and 
$i_Y : Y \rightarrow X$ denotes the inclusion, then the sets $(\mathcal{P} \cap Y)^*$ and 
$i_Y^*(\mathcal{P}^*)$ coincide in $Y^*$. 
\end{lemma}

\begin{proof}
Since
\begin{align*}
(\mathcal{P} \cap Y)^* &= \{{\boldsymbol\xi} \in Y^* : {\boldsymbol\xi}(\mathbf{v}) \leq 1 \text{ for all } \mathbf{v} \in \mathcal{P} \cap Y \} 
\textrm{ and} \\
i_Y^*(\mathcal{P}^*) &= \{i_Y^*({\boldsymbol\eta}) :  {\boldsymbol\eta}(\mathbf{w}) \leq 1 \text{ for all } \mathbf{w} \in \mathcal{P} \} \, ,
\end{align*}
we have the obvious inclusion  $i_Y^*(\mathcal{P}^*) \subset (\mathcal{P} \cap Y)^*$. The reverse inclusion is a consequence of the Hahn-Banach theorem: any linear functional ${\boldsymbol\xi}$ on the subspace $Y$ which satisfies ${\boldsymbol\xi}(\mathbf{v}) \leq 1$ on $\mathcal{P} \cap Y$ can be extended to a linear functional ${\boldsymbol\eta}$  on $X$ which satisfies ${\boldsymbol\eta}(\mathbf{w}) \leq 1$ on $\mathcal{P}$.
\end{proof}

\begin{proof}[Proof of Theorem~\ref{invariant-Groemer}]
In order to apply Groemer's theorem, we provide $X^*$ with an arbitrary Euclidean structure and engage in a simple translation exercise: to the one-dimensional subspace $W \subset X$ corresponds the one-dimensional subspace $W' := \Ann(W)^\bot \subset X^*$ consisting of those vectors in $X^*$ that are perpendicular to the annihilator of $W$. To a hyperplane $Y \subset X$ containing $W$ corresponds the hyperplane $Y' := \Ann(Y)^\bot$ containing $W'$. 

Now consider the diagram
 $$
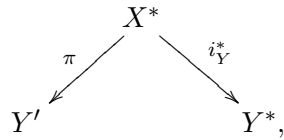
\xymatrix{
      &  X^* \ar[dl]_{\pi}\ar[dr]^{i_Y^*} &   \\
      Y' &     &  Y^*,}
$$
where $\pi$ is the orthogonal projection and the linear surjection $i_Y^*$ is the dual of the inclusion $i_Y : Y \rightarrow X$. The kernel of both maps is the annihilator of $Y$ and so, by Lemma~\ref{remark_1} the orthogonal projections of $\mathcal{P}^*$ and $\mathcal{Q}^*$ are translates in $Y'$ if and only if the same is true of their images 
$i_Y^*(\mathcal{P}^*)$ and $i_Y^*(\mathcal{Q}^*)$ in $Y^*$. Since Lemma~\ref{remark_2} tells us that 
$$
i_Y^*(\mathcal{P}^*) = (\mathcal{P} \cap Y)^* \textrm{ and } i_Y^*(\mathcal{Q}^*) = (\mathcal{Q} \cap Y)^* \, ,
$$
the hypotheses of Groemer's theorem are fulfilled and we conclude that $\mathcal{P}^*$ is a translate of $\mathcal{Q}^*$.
\end{proof}

The proof of Theorem~\ref{Minkowski-Randers-characterization} is a simple consequence of Theorem~\ref{invariant-Groemer} and the geometric characterization of Minkowski-Randers spaces given in Proposition~\ref{geometric-Minkowski-Randers}.

\begin{proof}[Proof of Theorem~\ref{Minkowski-Randers-characterization}]
Let $\mathbf{x}$ and $\mathbf{y}$ be two arbitrary distinct points in the affine space $V$ and let $\tau_{\mathbf{y-x}}$ be the translation by the vector $\mathbf{y-x}$. We must show that the unit co-disc  $\mathcal{P}^* := B_{\mathbf{x}}^*$ in $T_\mathbf{x}^*V$ is a translate of $\mathcal{Q}^* := \tau_{\mathbf{y-x}}^*(B_{\mathbf{y}}^*)$. 

In order to directly apply Theorem~\ref{invariant-Groemer}, let us denote by $X$ the tangent space $T_\mathbf{x} V$, let the one-dimensional subspace $W \subset X$ be the tangent space to the line $\mathbf{xy}$ at the point $\mathbf{x}$, and let $\mathcal{P}$ and $\mathcal{Q}$ be the convex bodies $B_\mathbf{x}$ and $D\tau_\mathbf{x-y}(B_\mathbf{y})$, respectively.

By hypothesis, the restriction of the Finsler metric in $V$ to every affine hyperplane is the sum of a translation invariant metric and a $1$-form. By Proposition~\ref{geometric-Minkowski-Randers} this means that
if $U \subset V$ is any affine hyperplane containing the line $\mathbf{xy}$ and $Y \subset X$ is its tangent space at $\mathbf{x}$, we have that $(\mathcal{P} \cap Y)^*$ is a translate of $(\mathcal{Q} \cap Y)^*$. Theorem~\ref{invariant-Groemer} implies that $\mathcal{P}^* = B_{\mathbf{x}}^*$ and $\mathcal{Q}^* = \tau_{\mathbf{y-x}}^*(B_{\mathbf{y}}^*)$ are translates. Since $\mathbf{x}$ and $\mathbf{y}$ are arbitrary, we can again apply Proposition~\ref{geometric-Minkowski-Randers} to conclude that the Finsler metric on $V$ is Minkowski-Randers.
\end{proof}

Using very simple arguments, we can substantially strengthen Theorem~\ref{Minkowski-Randers-characterization}: 

\begin{corollary}\label{Minkowski-Randers-dense-planes}
An $n$-dimensional affine space $V$ provided with a $C^2$ Finsler metric is a Minkowski-Randers space if and only if for some $k$ $(2 \leq k \leq n-1)$ the set of  $k$-planes in $V$ which inherit a Minkowski-Randers metric is dense in the affine Grassmannian $\Graff_k(V)$.
\end{corollary}

\noindent 
{\it Remark.} We will use the term {\sl $k$-plane} as an abbreviation for {\sl $k$-dimensional affine subspace.} It would have been more familiar to use the term $k$-flat, but we need to reserve the term
{\sl flat} for a totally geodesic submanifold of zero curvature. 

\begin{proof}
Let us start by noticing that if for some fixed $k$ $(2 \leq k \leq n-1)$ all $k$-planes in $V$ provided
with their induced metric are Minkowski-Randers spaces, then $V$ is a Minkowski-Randers space. Indeed, we have just proved this for the case $k = n-1$, and the general case is obtained by fixing $k$ and applying induction on
$n$ starting with $n = k + 1$.

We now show that if the set of $k$-planes which are Minkowski-Randers spaces with their inherited metric is dense in the affine Grassmannian of $k$-planes in $V$, then every $k$-plane in $V$ is a Minkowski-Randers space. 

Let $Y$ be a $k$-plane in $V$, let $F$ denote the Finsler metric it inherits from $V$, and let $\pi : V \rightarrow Y$ be an affine projection of $V$ onto $Y$ (e.g., along $(n-k)$-dimensional affine subspaces orthogonal to $Y$ for some Euclidean structure). If $\{Y_n : n \in \mathbb{N}\}$ is a sequence of $k$-planes that converge to $Y$, we may assume that all the maps
$$
\pi_n := \pi_{|Y_n} : Y_n \longrightarrow Y
$$
are affine isomorphisms.  Since the class of Minkowski-Randers metrics is preserved by affine isomorphisms,
if the induced metric on $Y_n$, $F_n$, is Minkowski-Randers, then $F'_n := \pi_n^{-1*} F_n$ is a Minkowski-Randers metric on $Y$. The convergence of the $k$-planes $Y_n$ to $Y$ in the affine Grassmannian $\Graff_k(V)$ and the continuity of the ambient metric implies that the $F'_n$ converge to $F$ uniformly on compact subsets and this implies that $F$ is Minkowski-Randers. 
\end{proof}

\section{Main theorem and applications} \label{Applications-section}

We are finally ready to prove our main results and discuss some related questions.  

\begin{maintheorem}{I} 
Let the distance function $d$ define a possibly asymmetric metric on $\R^n$ for which straight lines are geodesics. If the metric is invariant under a Euclidean crystallographic group $\Gamma$ and is derived from a $C^2$ Finsler metric $F$, then there exists a possibly asymmetric norm $\| \cdot \|$ and a function $f$ on $\R^n$ such that 
$$
d(\mathbf{x},\mathbf{y}) = \| \mathbf{y - x} \| + f(\mathbf{y}) - f(\mathbf{x})
$$ 
for every pair of points $\mathbf{x}$, $\mathbf{y}$ in $\R^n$.
\end{maintheorem}

\begin{proof}
We start by proving that $(\mathbb{R}^n, F)$ is a Minkowski-Randers space by showing that there is a dense
subset of $2$-planes in $\mathbb{R}^n$ on which the induced metric is Minkowski-Randers. 

By Bieberbach's first theorem, we may assume that $\Gamma$ contains the group $\mathbb{Z}^n$ of translations 
by integer vectors. If $Y$ is a $2$-plane of the form
$$
Y = \{ s\mathbf{m}_1 + t\mathbf{m}_2 + \mathbf{y}: s, t \in \mathbb{R} \}
$$
for some given linearly independent vectors $\mathbf{m}_1$ and~$\mathbf{m}_2$ in~$\mathbb{Z}^n$ and a  given
point $\mathbf{y} \in \mathbb{R}^n$, then $Y$ {\sl and its induced metric} are invariant under the $\mathbb{Z}^2$ action
$$
\left((p,q), \mathbf{x}\right) \longmapsto \mathbf{x} + p\mathbf{m}_1 + q\mathbf{m}_2 .
$$
Theorem~\ref{2D-case} then implies that the induced Finsler metric on $Y$ is the sum of a translation-invariant metric and a closed $1$-form. In particular $Y$ is a  Minkowski-Randers space. 

Since the class of $2$-planes that are parallel to a subspace spanned by two linearly independent integer vectors is dense in $\Graff_2(\mathbb{R}^n)$, Corollary~\ref{Minkowski-Randers-dense-planes} tells us that $(\mathbb{R}^n, F)$ is a Minkowski-Randers space and so there exists Minkowski norm $\| \cdot\|$
and a differential $1$-form $\beta$ on $\mathbb{R}^n$ such that 
$$
F(\mathbf{x},\mathbf{v}) = \|\mathbf{v}\| + \beta_\mathbf{x}(\mathbf{v}) .
$$
Moreover, since the pullbacks or restrictions of $\beta$ to a dense set of $2$-planes are closed, $\beta$ must be a closed $1$-form, and---given the regularity hypotheses in the theorem and in our definition of Minkowski norm---equal to the differential of a $C^3$ function $f: \mathbb{R}^n \rightarrow \mathbb{R}$. 

We conclude that if $d : \mathbb{R}^n \times \mathbb{R}^n \rightarrow [0, \infty)$ is the distance defined
by the Finsler metric $F$, then 
$$
d(\mathbf{x},\mathbf{y}) = \| \mathbf{y - x} \| + f(\mathbf{y}) - f(\mathbf{x})
$$ 
for every pair of points $\mathbf{x}$, $\mathbf{y}$ in $\R^n$.
\end{proof}

Note that Theorem~\ref{differential-geometric-version}, the differential-geometric formulation of this result, follows at once from Theorem~\ref{metric-version} by considering the universal cover and deck transformations of compact flat Riemannian manifolds. We can also extend Theorems~\ref{metric-version} and~\ref{differential-geometric-version} from Finsler metrics to $1$-densities by means of the following simple

\begin{lemma}
Let $(M,F)$ be a Finsler manifold and let $L: TM \rightarrow \mathbb{R}$ be a $1$-density with compact support. For all sufficiently small numbers $\epsilon > 0$, $F + \epsilon L$ is a Finsler metric. As a consequence, if $(\mathbb{R}^n, F)$ is a periodic Finsler metric and $L$ is a periodic $1$-density on $\mathbb{R}^n$, for all sufficiently small numbers $\epsilon > 0$, $F + \epsilon L$ is a periodic Finsler metric.
\end{lemma}

\begin{proof}
The proof is immediate from the compactness of $M$, the homogeneity of $L$ and $F$ and the conditions of
positivity and convexity that define a Finsler metric.
\end{proof}

\begin{maintheorem}{I${}^{\prime \prime}$}
Let $L(\mathbf{x}, \mathbf{v})$ be a continuous, real-valued function on $\mathbb{R}^n \times  \mathbb{R}^n $ that is $C^2$ on $\mathbb{R}^n \times  \mathbb{R}^n \! \setminus \! \{\mathbf{0}\}$ and which, on this open set, satisfies the system of differential equations
$$
\frac{\partial^2 L}{\partial x_i \partial v_j} = \frac{\partial^2 L}{\partial x_j \partial v_i} \ \ 
(1 \leq i, j \leq n) .
$$
If $L(\mathbf{x} + \mathbf{m}, \lambda \mathbf{v}) = \lambda L(\mathbf{x}, \mathbf{v})$ for all  $\lambda > 0$ and all integer vectors $\mathbf{m} \in \mathbb{Z}^n$, then there exists a $C^3$ function
$f : \mathbb{R}^n \rightarrow \mathbb{R}$ such that $L(\mathbf{x}, \mathbf{v}) - df(\mathbf{x})\cdot \mathbf{v}$
is independent of the $\mathbf{x}$-variable. 
\end{maintheorem}

\begin{proof}
Recall that, by Hamel's theorem, all the system of differential equations says is that the $1$-density $L$ is projective. If we add to $L$ a sufficiently large Euclidean metric, we obtain a projective Finsler metric which is invariant under integer translations. By Theorem~\ref{metric-version} this metric is the sum of a Minkowski norm and the differential of a $C^3$ function $f$ on $\mathbb{R}^n$. Since the Euclidean metric is invariant under translations, it follows that the difference $L - df$  must also be invariant under translations. 
\end{proof}

Besides indicating that the main results of this paper belong more properly to variational calculus than 
to Finsler geometry, the extension to periodic, projective $1$-densities gives us the flexibility to consider differences of Finsler metrics. For instance, the following characterization of projective $1$-densities of compact support immediately implies Theorem~\ref{determination-at-infinity} when applied to the difference of two projective Finsler metrics that agree outside a compact set.

\begin{theorem}\label{projective-densities-compact-support}
A $C^2$ projective $1$-density on $\mathbb{R}^n$ that vanishes outside a compact set is the differential of a $C^3$ function. 
\end{theorem}

\begin{proof}
Choose a lattice $\Lambda \subset \mathbb{R}^n$ such that the support of the $1$-density 
is contained in the tangent bundle of the interior of a fundamental domain. Extend the definition of the $1$-density by periodicity and add a sufficiently large Euclidean metric so as to obtain a Finsler metric $F$ on the torus $\mathbb{R}^n/\Lambda$. By Theorem~\ref{differential-geometric-version}, $F$ is the sum of a flat metric and a closed form. The flat metric agrees with the Euclidean metric outside
the support of the $1$-density and hence equals the Euclidean metric. It follows that the $1$-density, as seen on the torus, is a closed $1$-form and, therefore, it is the differential of a function when lifted to $\mathbb{R}^n$.
\end{proof}

Because many interesting Riemannian manifolds carry {\sl compact  flats} (i.e., compact, totally geodesic submanifolds of zero curvature), Theorem~\ref{differential-geometric-version} can be used to obtain some surprising results in geodesic rigidity. Just as in the Riemannian theory (for some of its most notable recent results see the papers~\cite{Matveev:2003} and~\cite{Matveev:2009}), it is useful to make a distinction between the equivalence relation comprising those metrics with the same unparameterized geodesics, and the more restrictive relation comprising those for which the midpoints of geodesic segments agree.  

\begin{definition}
We shall say that two complete Finsler metrics $F_1$ and $F_2$ on a manifold $M$ are {\sl projectively equivalent} if for each geodesic $t \mapsto \gamma(t)$  in $(M,F_1)$ there exists an orientation preserving diffeomorphism $\sigma : \mathbb{R} \rightarrow \mathbb{R}$ for which $t \mapsto \gamma(\sigma(t))$ is a geodesic in $(M,F_2)$. We shall say that the Finsler metrics are {\sl affinely equivalent} if we can always take
$\sigma(t) = \lambda t$, where the positive real number $\lambda$, like the reparameterization $\sigma$, depends on the geodesic. 
\end{definition}

Affine equivalence is a very restrictive condition. For instance, Busemann's construction yields many reversible Finsler metrics on $\mathbb{R}^n$ that are projectively equivalent to the Euclidean metric, but a Finsler metric that is affinely equivalent to the Euclidean metric must be a Minkowski norm (see Section~17 in~\cite{Busemann:1955}). Likewise, by a theorem of Szab\'o~\cite{Szabo:1981} (see also~\cite{Vincze:2005}), reversible Berwald Finsler metrics are precisely those reversible Finsler metrics that are affinely equivalent to a Riemannian metric. Theorem~\ref{differential-geometric-version} implies that  a $C^2$ reversible Finsler metric on a compact manifold is projectively equivalent to a flat Riemannian metric if and only if it is affinely equivalent to it. This result can be extended to a larger class of Riemannian manifolds: 

\begin{corollary}\label{compact-flats-rigidity}
If $(M,g)$ is a Riemannian manifold where every geodesic lies in a compact flat 
of dimension greater than one, then every reversible Finsler metric on $M$ that is projectively  equivalent to $g$ is in fact affinely equivalent to it. In particular, every reversible Finsler metric on $M$ that is projectively equivalent to $g$ is a Berwald metric.
\end{corollary}

\begin{proof}
If the reversible metric $F$ has the same unparametrized geodesics as $g$, then, by Theorem~\ref{differential-geometric-version}, every compact flat of dimension greater than one in $(M,g)$ will be a compact flat in $(M,F)$. Since any two flat (i.e., locally Minkowski) metrics on a manifold are affinely equivalent, and since every geodesic belongs to a common flat, the metrics $F$ and $g$ are affinely equivalent. 
\end{proof}

Explicit examples of Riemannian manifolds satisfying the hypothesis of Corollary~\ref{compact-flats-rigidity} are those that admit a finite covering of the form 
$$ 
T^k \times C_1 \times \cdots \times C_r \times S_1\times \cdots \times S_l \, ,
$$
where $T^k$ is a flat torus, $C_i$ is a simply-connected manifold all of whose geodesics
are closed, and $S_j$ is a simply-connected irreducible, compact symmetric space of  rank at least $2$ 
$(1 \leq i \leq r, 1 \leq j \leq l)$ with the obvious restrictions on $k$, $r$, and $l$. Conversely, Molina and Olmos proved in~\cite{Molina-Olmos:2001} that if a real-analytic, complete Riemannian manifold is such that every geodesic lies on a compact flat, then $M$ must admit such a covering.

\begin{maintheorem}{III} \label{symmetric-spaces}
Let $(M,g)$ be a compact Riemannian symmetric space of rank greater than one. If $F$ is a reversible $C^2$ Finsler metric on $M$ whose unparametrized geodesics coincide with those of $g$, then $F$ is invariant under the geodesic symmetries of $(M,g)$. In other words, $(M,F)$ is a symmetric Finsler space. 
\end{maintheorem}

\begin{proof}
Given that every geodesic of a compact Riemannian space of rank greater than one is contained
in a compact flat, Corollary~\ref{compact-flats-rigidity} tells us that $F$ is affinely
equivalent to the Riemannian metric $g$. We have then that $(M,F)$ is an {\sl affine symmetric Berwald metric} in the terminology of Szab\'o~\cite{Szabo:2006}, and, in point {\it (D)}, page~23 of this paper, Szab\'o shows that for such a metric the geodesic symmetries are isometries (see also Berestovskii's characterization of Finslerian symmetric spaces in~\cite[pp.164--166]{Berestovskii:1985}). 
\end{proof}

An easy consequence of Theorem~\ref{symmetric-spaces} we obtain the following:

\begin{corollary}~\label{Lie-Groups}
 If $F$ is a $C^2$ reversible Finsler metric on compact, connected Lie group $G$ which is neither $SU(2)$ nor $SO(3)$ and $F$ has the same unparametrized geodesics as a bi-invariant Riemannian metric, then it is a bi-invariant metric.
\end{corollary}

\begin{proof}
When $G$ is provided with a bi-invariant Riemannian metric it is a symmetric space and thus geodesic symmetries generate the connected component of its isometry group: left and right translations by  $G$. It follows from Theorem~\ref{symmetric-spaces} that $F$ is also invariant under left and right translations by $G$. 
\end{proof}

\section{Questions and remarks} \label{Questions-section}

We close the paper with some natural questions arising from our results. 

\subsection{Hilbert's fourth problem for projective manifolds.}

\begin{question}
Under what conditions does a compact projective manifold admit a continuous metric
or a $C^2$ Finsler metric for which straight lines are geodesics? 
\end{question}

In the continuous case, a generalization of this problem emphasizing systems of curves other than lines was treated by Busemann and Salzmann in~\cite{Busemann-Salzmann:1965} (see also Sections~14 and~15 in~\cite{Busemann:1970}). 

Compact projective manifolds obtained from divisible convex bodies admit a $C^1$ Finsler metric for which straight lines are geodesics (see~\cite{Crampon:2014} for a survey of the subject). On the other hand, we suspect the $C^2$ case to be very rigid. It is possible that if a compact projective manifold $(M,F)$ admits a $C^2$ reversible, projective Finsler metric, then either its fundamental group is finite or $(M,F)$ is affinely equivalent to a Riemannian space form. 

\subsection{A uniqueness theorem for projectively equivalent metrics}

\begin{question}
Let $F_1$ and $F_2$ be two projectively equivalent $C^2$ Finsler metrics defined on an open, geodesically convex set $\mathcal{O} \subset  \mathbb{R}^n$. Assume that both metrics agree outside a compact subset of $\mathcal{O}$. Is it true that $F_1$ and $F_2$ differ only by a differential? 
\end{question}

The first author has proved this generalization of Theorem~\ref{determination-at-infinity}  when $n = 2$, but the proof relies in an essential way on the strictly two-dimensional integral-geometric arguments of~\cite{Alexander:1978} and~\cite{Alvarez-Berk:2010}.   

\subsection{Projective equivalence to a symmetric space}

\begin{question}
If $F$ is a $C^2$ Finsler metric on a compact Riemannian symmetric space $(M,g)$ and the  unparametrized geodesics of $F$ coincide with those of $g$, is $F$ the sum of a Finsler metric that is invariant under the connected component of the isometry group of $(M,g)$ and a closed $1$-form?
\end{question} 

Note that if $F$ is such a metric, then its even part
$$
F_e(\mathbf{v_x}) := \frac{1}{2}(F(\mathbf{v_x}) + F(-\mathbf{v_x}))
$$
is a Finsler metric which is reversible and has the same unparametrized
geodesics as $F$ (and $g$). Theorem~\ref{symmetric-spaces} tells us that $F_e$ is invariant under the 
geodesic symmetries of $(M,g)$ and hence under the connected component of its isometry group. The difficulty is that we cannot say too much about the odd part, 
$$
F_o(\mathbf{v_x}) := \frac{1}{2}(F(\mathbf{v_x}) - F(-\mathbf{v_x})) .
$$
We do know, by Theorem~\ref{differential-geometric-version}, that the pullback of this $1$-density to every flat torus in $(M,g)$ is a translation-invariant $1$-density plus a closed $1$-form. On the other hand the main result in~\cite{Alvarez:2013} tells us that the pullback of $F_o$ to every Helgason sphere is a closed $1$-form.


\bibliography{paperbib}

\def\cprime{$'$}
\providecommand{\bysame}{\leavevmode\hbox to3em{\hrulefill}\thinspace}
\providecommand{\MR}{\relax\ifhmode\unskip\space\fi MR }
\providecommand{\MRhref}[2]{%
  \href{http://www.ams.org/mathscinet-getitem?mr=#1}{#2}
}
\providecommand{\href}[2]{#2}
\begin{thebibliography}{10}

\bibitem{Alexander:1978}
R.~Alexander, \emph{Planes for which the lines are the shortest paths between
  points}, Illinois J. Math. \textbf{22} (1978), no.~2, 177--190. \MR{MR490820
  (82d:53042)}

\bibitem{Alvarez:2003}
J.~C. {\'A}lvarez~Paiva, \emph{Hilbert's fourth problem in two dimensions},
  Mass Selecta: teaching and learning advanced undergraduate matematics
  (Providence, RI) (S.~Katok, A.~Sossinski, and S.~Tabachnikov, eds.), Amer.
  Math. Soc., 2003, pp.~165--183.

\bibitem{Alvarez:survey}
\bysame, \emph{Some problems on {F}insler geometry}, Handbook of differential
  geometry. {V}ol. {II}, Elsevier/North-Holland, Amsterdam, 2006, pp.~1--33.
  \MR{2194667 (2006k:53128)}

\bibitem{Alvarez:2013}
\bysame, \emph{Asymmetry in {H}ilbert's fourth problem}, arXiv preprint
  arXiv:1301.2524 (2013).

\bibitem{Alvarez-Berk:2010}
J.~C. {\'A}lvarez~Paiva and G.~Berck, \emph{Finsler surfaces with prescribed
  geodesics}, arXiv preprint arXiv:1002.0243 (2010).

\bibitem{Alvarez-Fernandes:2007}
J.~C. {\'A}lvarez~Paiva and E.~Fernandes, \emph{Gelfand transforms and
  {C}rofton formulas}, Selecta Math. (N.S.) \textbf{13} (2007), no.~3,
  369--390. \MR{MR2383600 (2009h:53172)}

\bibitem{Alvarez-Gelfand-Smirnov}
J.~C. {\'A}lvarez~Paiva, I.~M. Gelfand, and M.~Smirnov, \emph{Crofton
  densities, symplectic geometry and {H}ilbert's fourth problem}, The
  Arnold-Gelfand mathematical seminars, Birkh\"auser Boston, Boston, MA, 1997,
  pp.~77--92. \MR{98a:52005}

\bibitem{Berestovskii:1985}
V.~N. Berestovski\u\i, \emph{Generalized symmetric spaces}, Sib. Math. J.
  \textbf{26} (1985), no.~2, 159--170.

\bibitem{Busemann:1955}
H.~Busemann, \emph{{The Geometry of Geodesics}}, Pure and Applied Mathematics,
  vol.~6, Academic Press, New York, NY, 1955.

\bibitem{Busemann:1970}
\bysame, \emph{{Recent Synthetic Differential Geometry}}, Ergebnisse der
  Mathematik und ihrer Grenzgebiete, vol.~54, Springer-Verlag, New York
  Heidelberg Berlin, 1970.

\bibitem{Busemann:Hilbert}
\bysame, \emph{Problem {IV}: {D}esarguesian spaces}, Mathematical developments
  arising from Hilbert problems (Proc. Sympos. Pure Math., Northern Illinois
  Univ., De Kalb, Ill., 1974), Amer. Math. Soc., Providence, R. I., 1976,
  pp.~131--141. Proc. Sympos. Pure Math., Vol. XXVIII. \MR{55 \#3940}

\bibitem{Busemann-Salzmann:1965}
H.~Busemann and H.~Salzmann, \emph{Metric collineations and inverse problems},
  Mathematische Zeitschrift \textbf{87} (1965), no.~1, 214--240.

\bibitem{Crampon:2014}
M.~Crampon, \emph{The geodesic flow of {F}insler and {H}ilbert geometries},
  Handbook of Hilbert Geometry, IRMA Lectures in Mathematics and Theoretical
  Physics, vol.~22, European Mathematical Society, Zurich, 2014, pp.~161--206.

\bibitem{Groemer:1997}
H.~Groemer, \emph{On the determination of convex bodies by translates of their
  projections}, Geometriae Dedicata \textbf{66} (1997), no.~3, 265--279.

\bibitem{Hamel:1903}
G.~Hamel, \emph{{\"Uber die Geometrien in denen die Geraden die k\"urzesten
  sind}}, Math. Ann. \textbf{57} (1903), 231--264.

\bibitem{Matveev:2009}
V.~Kiosak and V.~S. Matveev, \emph{Complete {E}instein metrics are geodesically
  rigid}, Communications in Mathematical Physics \textbf{289} (2009), no.~1,
  383--400.

\bibitem{Matveev:2003}
V.~S. Matveev, \emph{Hyperbolic manifolds are geodesically rigid}, Inventiones
  mathematicae \textbf{151} (2003), no.~3, 579--609.

\bibitem{Molina-Olmos:2001}
B.~Molina and C.~Olmos, \emph{Decomposition of spaces with geodesics contained
  in compact flats}, Proceedings of the American Mathematical Society
  \textbf{129} (2001), no.~12, 3701--3709.

\bibitem{Papadopoulos:2014}
A.~Papadopoulos, \emph{Hilbert's fourth problem}, Handbook of Hilbert Geometry,
  IRMA Lectures in Mathematics and Theoretical Physics, vol.~22, European
  Mathematical Society, Zurich, 2014, pp.~391--431.

\bibitem{Pogorelov:1978}
A.V. Pogorelov, \emph{{The Minkowski multidimensional problem}}, Scripta Series
  in Mathematics, V.H. Winston \& Sons, Washington, D.C., 1978.

\bibitem{Szabo:1981}
Z.~I. Szab{\'o}, \emph{Positive definite {B}erwald spaces. {S}tructure theorems
  on {B}erwald spaces}, Tensor (NS) \textbf{35} (1981), 25--39.

\bibitem{Szabo:1986}
\bysame, \emph{Hilbert's fourth problem. {I}}, Adv. in Math. \textbf{59}
  (1986), no.~3, 185--301. \MR{MR835025 (88f:53113)}

\bibitem{Szabo:2006}
\bysame, \emph{Berwald metrics constructed by {C}hevalley's polynomials}, arXiv
  preprint math/0601522 (2006).

\bibitem{Vincze:2005}
C.~Vincze, \emph{A new proof of {S}zab{\'o}'s theorem on the
  {R}iemann-metrizability of {B}erwald manifolds}, Acta Math. Acad. Paedagog.
  Nyh{\'a}zi.(NS) \textbf{21} (2005), no.~2, 199--204.

\end{thebibliography}
\bibliographystyle{amsplain}

\end{document}